\documentclass[preprint,12pt]{elsarticle}

\usepackage{amssymb}
 \usepackage{amsthm}
\theoremstyle{theorem}
\newtheorem{theorem}{Theorem}

\newtheorem{proposition}{Proposition}
 
\theoremstyle{definition}

\journal{Discrete Mathematics}

\begin{document}

\begin{frontmatter}



\title{The smallest art gallery not guarded by every third vertex}


\author{Ralph Morrison}



\begin{abstract}
A polygonal art gallery can be observed by guards placed at one third of its corners.  However, the strategy of placing guards at every third corner does not work for all art galleries.  In this note, we provide an example of a nine-sided art gallery for which this strategy fails, and prove that this example is minimal.

\end{abstract}





\end{frontmatter}


\section{Introduction}
\label{sec:introduction}
Given an art gallery shaped like a polygon with $n$ sides, how many guards are needed to observe the whole gallery? Here we assume that each guard occupies a point $p$ of the polygon, and that any point $q$ in the polygon is visible to $p$ if and only if the line segment $\overline{pq}$ is contained in the polygon. This question was  posed by Victor Klee in 1973, and answered by  V\'{a}clav Chv\'{a}tal \cite{chvatal75}:   letting $\lfloor x\rfloor$ denotes the largest integer less than or equal to $x$, he showed that $\lfloor n/3\rfloor$ guards suffice, and are needed for some art galleries. For example, in a comb-shaped gallery like the one in Figure \ref{figure:comb}, each ``tooth'' requires one guard and contributes three vertices.

\begin{figure}[hbt]
\begin{centering}
\includegraphics{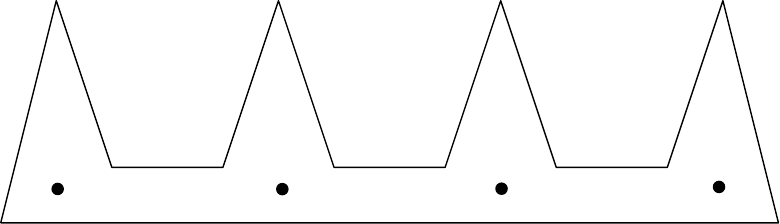}
\caption{An art gallery shaped like a comb requires $\lfloor n/3\rfloor$ guards}
\label{figure:comb}
\end{centering}
\end{figure}

 Steve Fisk offered a simplified proof in  \cite{fisk1978}, using the following $3$-coloring argument:  triangulate the polygon, and color the vertices of the polygon with three colors so that each triangle has one vertex of each color.  Choose the color that appears least often, and place a guard at each vertex of that color.  This uses at most $\lfloor n/3\rfloor$ guards.  Each triangle will have a guard on one of its vertices, and since the guard can observe the whole triangle, every triangle, and thus the entire polygon, is covered.

\begin{figure}[hbt]\begin{centering}
\includegraphics[scale=0.55]{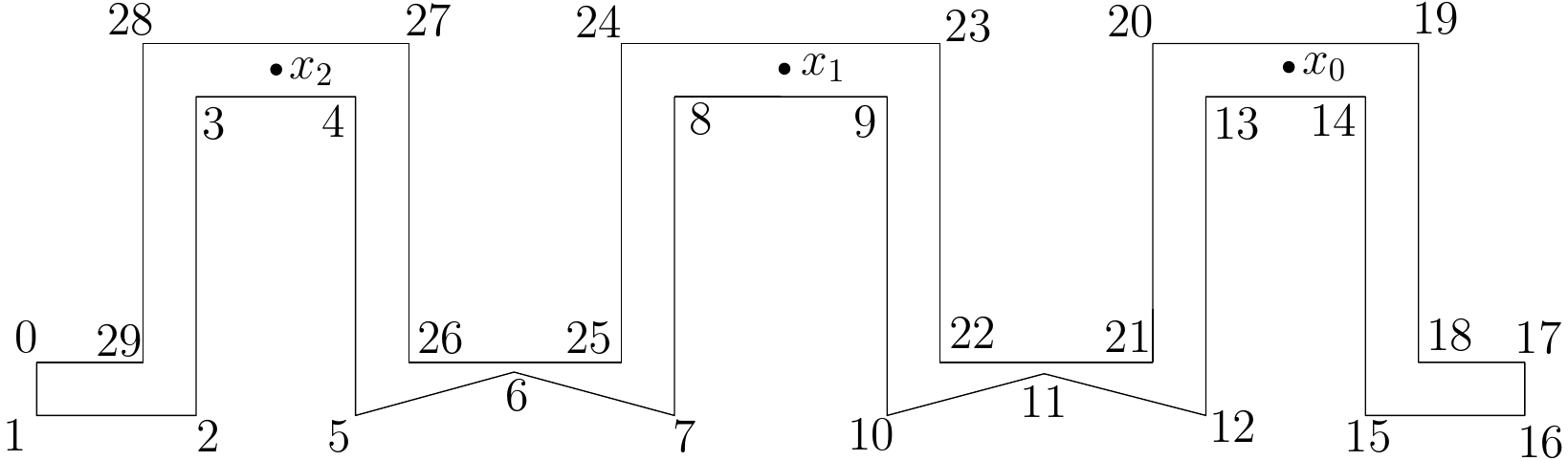}
\caption{A 30-gon that is not guarded by every third vertex}
\label{figure:30gon}
\end{centering}
\end{figure}

When $n$ is a multiple of $3$, one might be tempted to place a guard at every third vertex, for some choice of starting vertex.  However, this method of placing guards does not always work:  Joseph O'Rourke \cite[\S 1.2]{orourke1987} presents a counterexample with $n=30$ sides, reproduced in Figure~\ref{figure:30gon}.  Regardless of the starting vertex, placing a guard at every third vertex will fail to cover one of $x_0$, $x_1$, and $x_2$.

\begin{figure}[hbt]\begin{centering}
\includegraphics[scale=0.45]{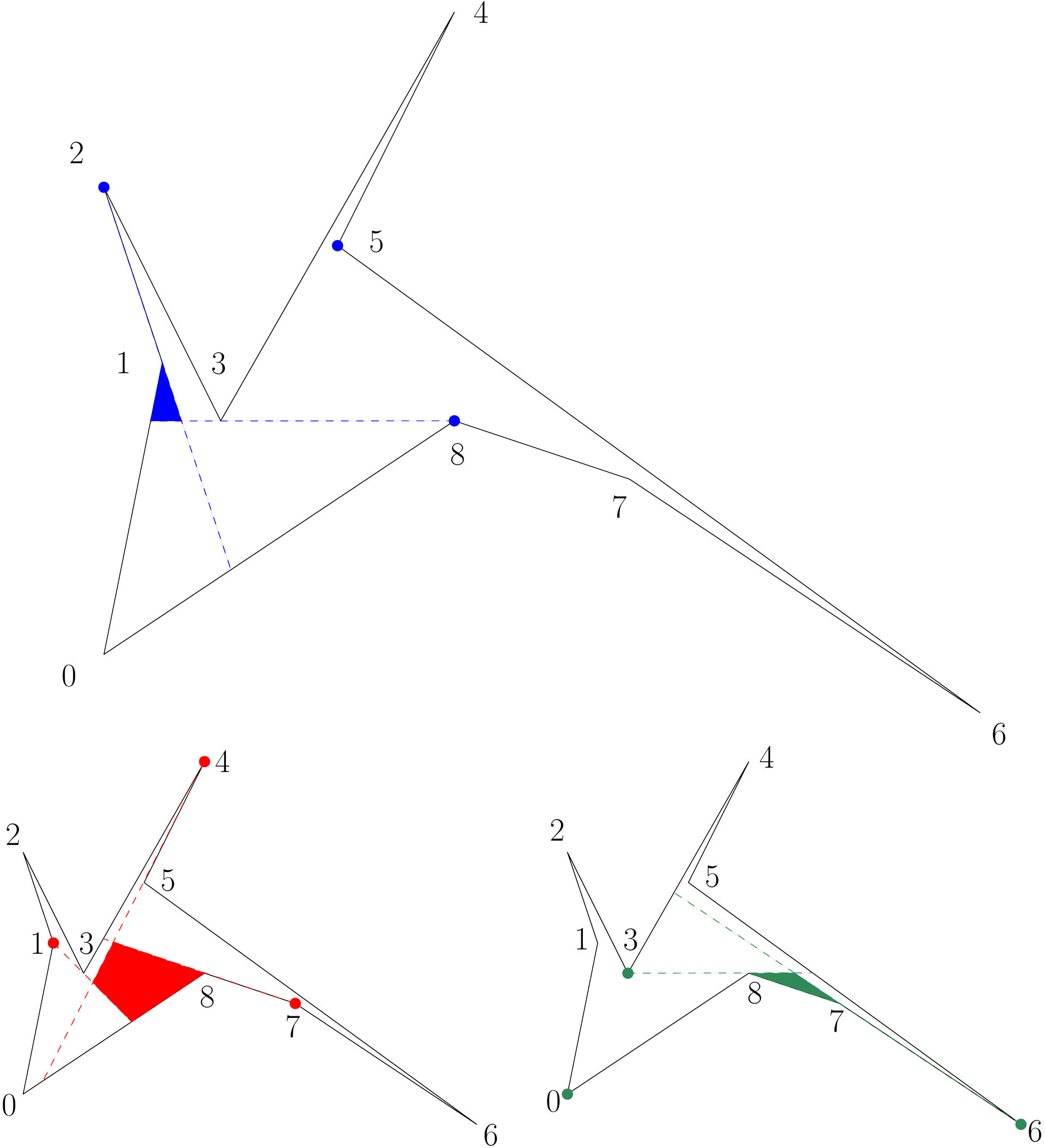}
\caption{A nonagon not guarded by every third vertex, with the blindspots shaded}
\label{figure:nonagon}
\end{centering}
\end{figure}

We present a smaller polygon that is not guarded by every third vertex, no matter which starting vertex we choose: the nonagon illustrated in Figure \ref{figure:nonagon}.  The vertices of this polygon have coordinates $(0,0)$, $(1,5)$, $(0,8)$, $(2,4)$, $(6,11)$, $(4,7)$, $(15,-1)$, $(9,3)$, and $(6,4)$.  If we label the vertices cyclically $0$ through $8$, then placing guards at every third vertex starting from $0$, from $1$, or from $2$ will fail to cover the whole polygon, as illustrated.

In Section \ref{sec:construction} we outline the method that constructed the nonagon.  In Section \ref{sec:allpolygons}, we argue that any polygon with fewer than nine sides is successfully covered by placing guards at every third vertex, meaning that this nonagon gives us the smallest art gallery not guarded by every third vertex.  In Section 4 we discuss related open problems.

\section{Constructing the nonagon}
\label{sec:construction}

In this section we provide the heuristic argument that lead to the construction of our nonagon.  

Suppose that $P$ is a nonagon that is not guarded by every third vertex.  Color the vertices of the polygon cyclically using three colors (red, blue, and green, going clockwise), and then triangulate the polygon.  By our assumption, placing guards at the red vertices does not cover the whole polygon, meaning there must be a triangle $T$ with only blue and green vertices; otherwise, the red vertices would guard each triangle, and thus the whole polygon.  (In our example, this triangle could have vertices $3$, $5$, and $8$.)

Suppose all the vertices  of $T$ are the same color, say blue.  Then between each two vertices of $T$ there would have to be two vertices of $P$, one green and one red. This means $P$ would be the union of the triangle $T$ and four quadrilaterals, each of which has two adjacent blue vertices.  But any two adjacent vertices of a quadrilateral guard the whole quadrilateral, meaning that the blue vertices must guard all of $P$, a contradiction.

Thus, it must be that $T$ has a mix of blue and green vertices. Without loss of generality, we may assume it has two blue and one green. We will now further assume that $T$ has all its edges in the interior of the polygon, meaning that none of its three vertices are joined by an edge of $P$.  
Labelling the vertices of $T$ clockwise as $G$, $B_1$, and $B_2$, the only  possible color sequence of the vertices of $T$ (starting at $G$) is:
$$GrB_1grB_2grb.$$
This is because the other two possible sequences, namely
$$GrB_1grbgrB_2$$
and
$$GrbgrB_1grB_2,$$
would make $T$ have a boundary edge, connected $B_2$ and $G$.

Thus $P$ will consist of the triangle $T$ glued to three other polygons:  a triangle colored $GrB_1$, a quadrilateral colored $B_1grB_2$, and a pentagon colored $B_2grbG$.  The new triangle can't contain any blindspots regardless of which color we choose, since it is convex and has a vertex of each color.  The quadrilateral can't contain a blue blindspot since it has two adjacent blue vertices, but it could contain a green blindspot.  Finally, the pentagon is the only portion of $P$ that could contain a blue blindspot.  Guided by these constraints, we can construct  the nonagon in Figure~\ref{figure:nonagon}.

\section{When $3$ doesn't divide $n$}
\label{sec:allpolygons}

\begin{figure}[hbt]\begin{centering}
\includegraphics{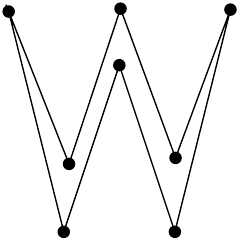}
\caption{An octagon not guarded by $2$ vertices placed three apart}
\label{figure:bad_octagon}
\end{centering}
\end{figure}

There are two reasonable interpretations of ``every third vertex'' if $n$ is \emph{not} a multiple of $3$.  Both versions place a guard at a vertex and, traveling either clockwise or counterclockwise, place a new guard at every third vertex.  However, one might choose to stop upon reaching the upper bound of $\lfloor n/3\rfloor$ guards; or one might stop just before passing the original guard, placing the final guard less than $3$ edges away from the first.  Under the former interpretation, there exists a polygon with $8$ sides where this strategy always fails to guard the polygon, as illustrated in Figure \ref{figure:bad_octagon}.  However, there does not exist such a polygon with $7$ sides, as shown in the following Proposition.

\begin{proposition} All heptagons are guarded by the ``every third vertex'' strategy with $2$ guards.
\label{prop:heptagons}
\end{proposition}

\begin{proof}

\begin{figure}[hbt]\begin{centering}
\includegraphics[scale=0.6]{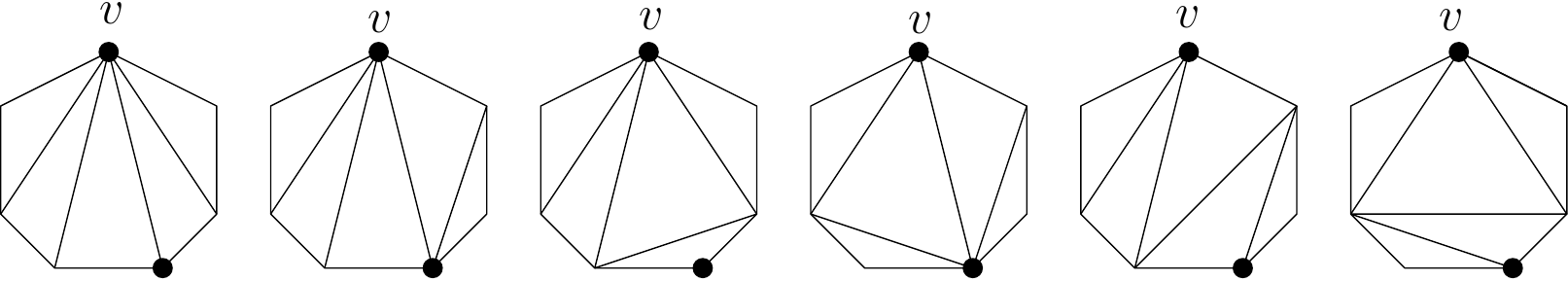}
\caption{Possible triangulations of a heptagon, with guarding vertices marked}
\label{figure:heptagons}
\end{centering}
\end{figure}

Let $P$ be a heptagon.  Triangulate it, resulting in five triangles and four diagonals.  Treating this triangulation as a graph, there are a total of $11$ edges, so the sum of the degrees of the vertices is $22$.  By the Pigeonhole Principle, some vertex $v$ must be connected to at least $\lceil 22/7\rceil=4$ other vertices.  So we have $\deg(v)=6$, $\deg(v)=5$, or $\deg(v)=4$. Figure \ref{figure:heptagons} illustrates the different possible triangulations of $P$ from the perspective of $v$, up to reflection.  In all cases, $v$ together with a vertex three edges away guards every triangle, and thus the entirety of $P$.

\end{proof}

We now choose the more generous interpretation:  that we place guards at every third vertex, stopping just before we pass the initial guard.  In this case we have the following result.

\begin{theorem}  Let $P$ be an $n$-gon, where $n\leq 8$.  Then for some choice of starting vertex, placing a guard at every third vertex of $P$ covers all of $P$.
\label{theorem:main}
\end{theorem}

Before we prove Theorem \ref{theorem:main}, we summarize a few helpful results for polygons.  See \cite[\textsection 1.2]{do2011} for more details.  An \emph{ear} of a polygon $P$ is a vertex $p_i$ such that the line segment $\overline{p_{i-1}p_{i+1}}$ connecting the two adjacent vertices is contained in the interior of $P$.  If a polygon with $n$ sides has an ear $p_i$, we can construct a polygon with $n-1$ sides by removing the vertex $p_i$ and connecting $p_{i-1}$ to $p_{i+1}$.  We refer to this construction as \emph{removing} the ear $p_i$ from the polygon.  We say two ears $p_i$ and $p_j$ are \emph{non-overlapping} if the interiors of the triangles with vertices at $p_{i-1}$, $p_i$, $p_{i+1}$ and at $p_{j-1}$, $p_j$, $p_{j+1}$ do not intersect.  See Figure \ref{figure:ears} for an example.

\begin{figure}[hbt]
\begin{centering}
\includegraphics{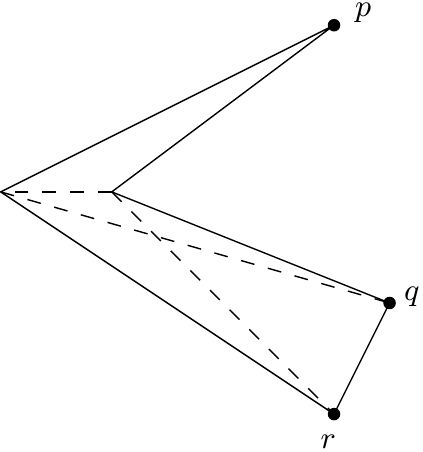}
\caption{A polygon with three ears, labelled $p$, $q$, and $r$.  The ears $p$ and $q$ are non-overlapping, as are $p$ and $r$, but $q$ and $r$ overlap}
\label{figure:ears}
\end{centering}
\end{figure}

\begin{theorem}[The two ears theorem]  If a polygon $P$ has $n\geq 4$ sides, then $P$ has at least two non-overlapping ears.  Moreover, any triangulation of $P$ reveals at least two non-overlapping ears.
\end{theorem}

This was first proven with the ``ear'' terminology in \cite{meisters1975}, but was also proven by Max Dehn around 1899 in an unpublished manuscript on the Jordan curve theorem, as discussed in \cite{guggenheimer1977}.  One quick proof of this result is to triangulate the polygon and look at the \emph{weak dual graph} of the triangulation, which has a vertex for each triangle and has two vertices joined if their triangles share an edge.  This graph will be a tree, and as such must have at least two leaves, which correspond to ears.

 The two ears theorem implies that if $n\geq 5$, then we can remove \emph{two} ears from an $n$-gon, resulting in a polygon with $n-2$ sides.  (If $n=4$, removing two non-overlapping ears results in a line segment.)  
  In our proof below, we will use the following observation:  if $p_i$ is an ear, then the triangle with vertices at $p_{i-1}$, $p_i$, and $p_{i+1}$ is always observed if we have placed a guard on every third vertex.  This is because at least one of $p_{i-1}$, $p_i$, and $p_{i+1}$ will be the location of a guard.  Therefore if we place guards on every third vertex of a polygon $P$, to show that they guard all of $P$ it suffices to show that they guard a subpolygon of $P$ obtained by removing ears.

\begin{proof}[Proof of Theorem \ref{theorem:main}]  Suppose $n\leq 5$.  By Chv\'{a}tal's theorem, there exists a vertex $v$ such that a single guard placed at $v$ can see the whole polygon. Placing guards at every third vertex starting at $v$ thus covers $P$.

Suppose $n=6$ or $n=7$. By the two ears theorem, we may remove ears from $P$ until it is a pentagon.  Again by Chv\'{a}tal's theorem, there exists a vertex $v$ that can see the whole pentagon.  Place a guard at every third vertex of $P$, starting with $v$.  By the argument preceding this proof, the fact that these guards cover the pentagon means that they cover all of $P$.  (Alternatively for $n=7$, this result follows from Proposition~\ref{prop:heptagons}.)

\begin{figure}[hbt]
\begin{centering}
\includegraphics{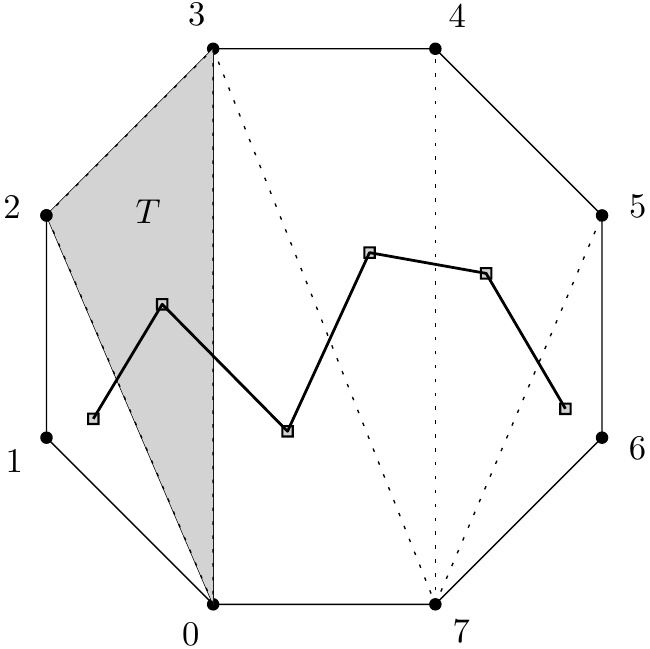}
\caption{A triangulation of an octagon, with the weak dual graph being a line segment}
\label{figure:octagon_dual}
\end{centering}
\end{figure}

Finally, suppose $n=8$.  Label the vertices of $P$ cyclically with the integers from $0$ through $7$. Consider any triangulation of $P$, and look at the weak dual graph.  If the weak dual graph has three or more leaves, then $P$ has three non-overlapping ears.  We may remove these to obtain a pentagon, and then we apply the same argument from the $n=6$ and $n=7$ cases.  Otherwise, the weak dual graph must have only two leaves (corresponding to two ears $p$ and $q$) and thus must be a line segment.  Choose a triangle $T$ corresponding to a vertex in the weak dual that is adjacent to a leaf.  Then in the triangulation, $T$ borders an ear, and has an edge on the boundary of $P$.  This is illustrated in Figure \ref{figure:octagon_dual}, with $p$ and $q$ at vertex $1$ and vertex $6$, respectively.

Up to relabelling and reorienting, the vertices of $T$ are at the vertices of $P$ labelled $0$, $2$, and $3$.  There are only two ways that placing guards at every third vertex of $P$ will fail to place a guard at some vertex of $T$:  putting the guards at $\{1,4,7\}$ and putting the guards at $\{4,1,6\}$.

Remove the two ears $p_1$ and $p_2$ from $P$ to obtain a hexagon $Q$.  Then remove the triangle $T$ from $Q$ to obtain a pentagon $R$.
  Let $v$ be a vertex such that all of $R$ is visible from $v$.  There are three ways to place the guards on every third vertex of $P$ so that a guard is on $v$:  starting at $v$ and proceeding clockwise, starting at $v$ and proceeding counterclockwise, or placing the guards so that the second guard is placed on $v$.  At least one of these ways places a guard at a vertex of $T$, since there are only two placements that fail to do so.  Thus, we may place the guards so that one is at $v$ and one is at a vertex of $T$.  These guards cover the pentagon $R$ and the triangle $T$, and thus the hexagon $Q$.  Since $Q$ was obtained from $P$ by removing ears, it follows that the guards cover $P$.
\end{proof}

\section{Open problems}
There are many generalizations of the art gallery theorem that suggest a strategy for placing guards might work, only to have it fall apart for certain counterexamples.  We may ask for \emph{minimal} counterexamples, similar to the one provided by our nonagon.

\begin{itemize}
\item  One variant of the art gallery problem is \emph{the fortress problem}:  given a polygonal fortress, place guards on its boundary so that they may observe every point \emph{exterior} to the fortress.  O'Rourke and Wood \cite[Theorem 6.1]{orourke1987} showed that for a fortress with $n$ sides, $\lceil n/2\rceil$ guards (all placed at vertices) will suffice, and are sometimes necessary.  However, the strategy of placing a guard at every second vertex does not always work:  they present a $20$-sided polygon where this strategy fails.  To this author's knowledge, it is unknown if this counterexample is minimal.

\item  Rather than guarding two-dimensional polygons, one could try to guard three-dimensional polyhedra.  However, there is a remarkable difference in this generalization:  \emph{not every polyhedron is covered by guards placed at its vertices}.  T. S. Michael \cite{michael2009} presents an example called the octoplex, which has 30 faces and 56 vertices. Placing $56$ guards, one on each vertex, fails to cover the whole polytope. It seems to be open whether there exists a polytope either with fewer vertices or with fewer faces that cannot be guarded by its vertices.
\end{itemize}

\section*{Acknowledgement}  The author would like to thank the students enrolled in Spring 2018's ``Discrete Geometry'' course at Williams College for many helpful polygonal discussions: Will Barrett, Spencer Brooks, Akhil Dayal, Max Everett, Will Howie, Will Huang, Rishad Karim, Ben Morton, Ben Mygatt, Jacob Pesikoff, and Amy Qiu.  

\end{document}